%% file: main.tex
\documentclass[a4paper]{article}

\usepackage{anysize}
\usepackage{multicol}
\usepackage{amsthm}
\usepackage{amsmath}
\usepackage{amssymb}
\usepackage{amsfonts}
\usepackage{amsxtra} 
\usepackage{mathrsfs}
\usepackage{color}
\usepackage{graphicx}
\usepackage[all]{xy}
\usepackage{enumerate}
\usepackage{stmaryrd}
\usepackage{booktabs}
\usepackage{tikz}
\usetikzlibrary{backgrounds, positioning, arrows}
\usepackage{mdframed, enumerate}

\marginsize{3cm}{3cm}{3.2cm}{3.2cm}

\usepackage{mathtools}

\title{Combinatorial Properties of the Raisonnier Filter}
\author{Spyridon Dialiatsis\thanks{University of Amsterdam (Institute of Logic, Language and Computation)}, Yurii Khomskii\thanks{University of Amsterdam (Amsterdam University College); and Universit\"at Hamburg (Fachbereich Mathematik)}}

\input{newcommands}

\renewcommand{\PP}{{\mathscr{P}}}

\renewcommand{\L}{L}

\newcommand{\BC}{\mathsf{BC}}
\newcommand{\dualBC}{\mathsf{dBC}}

\newcommand{\GAMMA}{{\boldsymbol{\Gamma}}}
\newcommand{\rationals}{\mathbb{Q}}

\newcommand{\Raisonnier}[1]{\mathcal{F}_{#1}}
\DeclarePairedDelimiterX{\set}[1]{\{}{\}}{\setargs{#1}}
\NewDocumentCommand{\setargs}{>{\SplitArgument{1}{\st}}m}
{\setargsaux#1}
\NewDocumentCommand{\setargsaux}{mm}
{\IfNoValueTF{#2}{#1} {#1\;\delimsize|\;\mathopen{}#2}}
\newcommand{\bigunion}[3]{{\bigcup_{#2}^{#3}{#1}}}
\DeclareMathOperator{\powerset}{\PP}
\newcommand{\powersetfin}[1]{\left[ #1 \right]^{<\omega}}

\newcommand{\almostsubseteq}{\subseteq^*}

\newcommand{\concat}{{^\frown}}

\DeclareMathOperator{\partitioningreals}{\mathcal{P}}
\newcommand{\goesthrough}{\in^*}
\newcommand{\slaloms}[1]{\mathcal{S}_{#1}}
\newcommand{\binaryslaloms}[2]{\mathcal{B}^{#1}_{#2}}
\DeclareMathOperator{\nat}{nat}
\DeclareMathOperator{\bin}{{\color{blue}BIN}}
\newcommand{\partreal}[1]{{\tilde{#1}}}

\newcommand{\ProjectiveSigma}[1]{\mathbf{\Sigma}^1_{#1}}

\newcommand{\rapidideal}{\mathcal{R}_r}
\newcommand{\unboundedideal}{\mathcal{R}_u}

\newcommand{\nullideal}{\mathcal{N}}
\newcommand{\meagreideal}{\mathcal{M}}
\newcommand{\nulladditiveideal}{\mathcal{N\mkern-2mu A}}
\newcommand{\meagreadditiveideal}{\mathcal{MA}}
\newcommand{\stronglynullideal}{\mathcal{SN}}
\newcommand{\stronglymeagreideal}{\mathcal{SM}}
\newcommand{\countableideal}{\mathcal{I}_{\textrm{ctbl}}}
\newcommand{\boundingnum}{\mathfrak{b}}
\newcommand{\dominatingnum}{\mathfrak{d}}

\renewcommand{\emptyset}{\varnothing}

\renewcommand{\restriction}{\mathord{\upharpoonright}}

\newcommand{\existinf}{\exists^\infty}
\newcommand{\forallbutfin}{\forall^\infty}

\DeclarePairedDelimiter{\floor}{\lfloor}{\rfloor}

\newcommand{\eqev}{=^*}
\newcommand{\leqev}{\leq^*}

\begin{document}

\maketitle
 
\begin{abstract} The \emph{Raisonnier filter} is a combinatorial object isolated by Jean Raisonnier in \cite{RaisonnierFilter} in order to simplify Shelah's proof that if all $\SIGMA^1_3$ sets are Lebesgue-measurable then there is an inner model with an inaccessible cardinal.  In this paper, we study the combinatorics of a general version of the Raisonnier filter, with an eye to potential applications in descriptive set theory. We prove various results about the Raisonnier Filter generated by simple sets, and we study the ideals associated to  the Raisonnier Filter being rapid and/or unbounded, connecting it to the null-additive and meagre-additive ideals studied by Pawlikowski and Shelah. Among the more interesting results is   a converse to  Raisonnier's theorem which indicates  that Raisonnier's technique is unlikely to lead to similar results for other regularity properties.  We also draw some conclusions regarding the Borel and dual Borel Conjectures. \end{abstract}

\section{Introduction}
 
The \emph{Solovay Model},  obtained by collapsing an inaccessible cardinal to $\omega_1$ using the L\'evy collapse, is a well-known model of set theory in which all projective sets of reals are Lebesgue measurable,  have the Baire property, and satisfy almost all other known regularity properties.  This construction   from   1970 was a celebrated early achievement showing what the recently-discovered forcing method was capable of \cite{SolovayAModel}.  However, it left the question open whether the inaccessible cardinal was really needed to produce a model with this property. This was settled by Shelah in 1984 when he proved---contrary to expectations---that an inaccessible cardinal was \emph{not} needed for the Baire property, but it \emph{was} needed for Lebesgue measurability \cite{ShelahInaccessible}.  

Shelah's proof of the latter result was simplified by Jean Raisonnier \cite{RaisonnierFilter} using purely combinatorial techniques, and this is the preferred exposition of Shelah's theorem to modern audiences.  Raisonnier's argument can be outlined as follows:

\begin{enumerate}
\item Assume that all $\SIGMA^1_3$ sets of reals are Lebesgue measurable, but $\aleph_1^{\L[a]} = \aleph_1$ for at least one parameter $a \in \ww$.
\item Define a special filter   $\F$ on $\omega$ using the reals of $\L[a]$, and show that it is $\SIGMA^1_3$.
\item Use the assumption that there are $\aleph_1$-many reals in $\L[a]$   to show that $\F$ is a non-trivial (proper) filter.
\item Use the weaker assumption that all $\SIGMA^1_2$ sets of reals are Lebesgue measurable to show that $\F$ is a \emph{rapid filter}.
\item By a classical result of Talagrand \cite{Talagrand},  a non-trivial rapid filter cannot be Lebesgue-measurable.  This contradicts the assumptions.
\item We conclude that if all $\SIGMA^1_3$ sets are Lebesgue measurable,  then   $\forall a \in \ww (\aleph_1^{\L[a]} < \aleph_1)$,  which is known to imply  that   $\aleph_1$ is an inaccessible cardinal in $\L$.
\end{enumerate}

 The filter $\F$ from above is often referred to as the \emph{Raisonnier Filter}.  In spite of the importance and elegance of this construction, the Raisonnier Filter itself has received relatively little attention in the literature outside the context of this proof. The only application we are aware of is a minor result on the separation of regularity properties for $\DELTA^1_3$- and $\DELTA^1_4$-levels in \cite{CichonPaper}.

This paper is motivated by the need to understand the combinatorics of this filter better,  with an eye on potential applications to other problems in descriptive set theory.  Indeed, some of our more interesting results provide new connections between the Raisonnier Filter and the null-additive and meagre-additive ideals, with some implications about the Borel and Dual Borel Conjecture.

We note that many questions about the consistency strength of projective regularity properties are still open. The most famous and long-standing of these is probably the question asked by Mathias \cite{HappyFamilies} and many others: \emph{what is the consistency strength of the assumption that all projective sets have the Ramsey property}? It is tempting to try to use the Raisonnier filter to solve this problem; however,  one of our results (Theorem \ref{thm rapid ideal and domination characterization of Sigma12 measurability}) provides a partial converse to point 4 in the argument above,  indicating that the Raisonnier filter is intimately tied to Lebesgue-measurability and is unlikely to apply to other regularity properties in this way.

\bigskip This paper is structured as follows: we  provide the necessary definitions, notation and context in Section \ref{sectionDEF}, and in Section \ref{sectionEASY} we look at properties of the Raisonnier filter $\F_X$ generated by simple subsets $X$ of the Cantor Space.  

In Section \ref{sectionUNBOUNDED} we look at the question whether the Raisonnier Filter can be unbounded (non-meagre) and consider a corresponding $\sigma$-ideal on the Cantor space.  Section \ref{sectionMAIN} contains one of the more important results, Theorem \ref{thm Raisonnier rapid filter characterization}, characterizing those $X$  that generate a rapid Raisonnier filter as precisely the \emph{null-additive} sets.  This characterization is used in Section \ref{sectionCONVERSE} to obtain a converse to Raisonnier's theorem,  Theorem \ref{thm rapid ideal and domination characterization of Sigma12 measurability}.  Section \ref{sectionBOREL} is devoted to various observations related to the Borel and Dual Borel Conjectures, and the isolation of new and potentially interesting conjectures, and \ref{sectionQUEST} ends with some problems.

\bigskip
  
The work in this paper was carried out by the first author in the context of his Master's Thesis \cite{SpyrosThesis}. 
Some details omitted here for readability can be found in the thesis.

\bigskip
  
\s{Acknowledgements}

\medskip 

We would like to thank Tan \"Ozalp for drawing our attention to the connection between binary slaloms and the null-additive ideal.

This scientific paper is supported by the Onassis Foundation -- Scholarship ID: F ZU 007-1/2024-2025.

\s{Declaration}

\medskip The research conducted in this work was not assisted by Large Language Models or Generative AI Models in any way.

\section{Definitions and Basic Context} \label{sectionDEF}

Our notation and conventions follow the standard in modern set theory. We will often identify subsets of $\omega$ and elements of $\dw$ via characteristic functions, and we will take it for granted that subsets of $\PP(\omega)$, such as filters on $\omega$, can be viewed as sets of reals and have a descriptive complexity.  $\rationals$ denotes the ``rationals'' in the Cantor Space, i.e., those  elements of $\dw$ whose values are eventually $0$. The notation $x =^* y$ denotes the statement that $x$ and $y$ are eventually equal,  $x \leq^* y$ that $y$ eventually dominates $x$, and $x \subseteq^* y$ that $x$ is a subset of $y$ except for finitely many members.

The term \emph{filter} always refers to filters on $\omega$.  We call a filter  \emph{trivial} if it contains every subset of  $\omega$ and \emph{Fr\'echet} if it contains exactly the co-finite subsets.

\begin{Note} Deviating slightly from common usage, the trivial filter $\PP(\omega)$ will by default be included in our subsequent definitions. For example, the trivial filter will be considered both a rapid and an unbounded filter, and we will explicitly mention when this is not desired.
\end{Note}

In   this paper we will often use partitions of $\omega$ into non-empty finite intervals, and we will use the following notation:
\begin{Def}
    A real $d \in \ww$ is called a \textit{partitioning real} if   it is strictly increasing and $d(0) = 0$. Let $\partitioningreals$ denote the set of partitioning reals.
    We will use the notation $I^d_n$ to abbreviate the interval $[d(n),d(n+1)) = \{k \in \omega \mid d(n) \leq k < d(n+1)\}$.
\end{Def}

Mokobodzki, Talagrand and others have considered rapid and unbounded filters, and the following definitions and characterizations are well-known.

\begin{Def} \label{def:unbounded}
    A filter $\mathcal{F}$ on $\omega$ is an \emph{unbounded filter} (also called \emph{non-meagre filter}) if the increasing enumerations of its elements form an unbounded family.
\end{Def}


\begin{Thm}[Talagrand, 1980 \cite{Talagrand}, see also {\cite[pp.~205-206]{BaJu95}}]
    \label{thm unbounded filter characterization}
    The following are equivalent for a filter $\mathcal{F}$ on $\omega$:
    \begin{enumerate}
        \item $\mathcal{F}$ is unbounded,
        \item $\mathcal{F}$ does not satisfy the property of Baire,
        \item for every increasing $f \in \dw$ there exists  $a \in \mathcal{F}$ so that $a \cap [n,f(n)] = \emptyset$ for infinitely many $n \in \omega$, and
        \item for every $d \in \partitioningreals$ there exists an $a \in \mathcal{F}$ such that $a \cap I^d_n = \emptyset$ for infinitely many $n \in \omega$.
    \end{enumerate}
\end{Thm}
\smallskip

\begin{Def}[Mokobodzki \cite{mokobodzki1967}]  \label{def rapid filter}
A filter $\F$ on $\omega$ is a \emph{rapid filter} if for every increasing $f \in \ww$ there exists $a \in \F$ such that $|a \cap f(n)| \leq n$ for all $n$.
\end{Def}

It is easy to  see that   $\F$ is rapid if and only if the set of increasing enumerations of its members forms a \emph{dominating family} in $\ww$. The precise bound in the definition is also not relevant, with some very weak restrictions, as is seen below.

\begin{Lem} \label{lem rapid filter alternative definition}
Let $\F$ be a filter and let $\varphi: \omega \to \omega$ be an increasing, unbounded function such that $\varphi(0) = 0$. Then the following are equivalent:
\begin{enumerate}
\item $\F$ is rapid, and
\item for every $f \in \partitioningreals$ there exists $a \in \F$ such that $|a \cap f(n)| \leq \varphi(n)$ for all $n$. 
\end{enumerate}
\end{Lem}

\begin{proof} The direction $\text{(2)} \Rightarrow \text{(1)}$ easily follows from \cite[Lemma 4.6.2]{BaJu95}.
    For the other direction,
    assume that $\mathcal{F}$ is rapid, and let $\varphi \in \ww$ have the properties defined above.
    Given an $m \in \omega$, let $N_m = \set{k \in \omega \st \varphi(k) \leq m}$.
    Let $f \in \partitioningreals$ and define the function $g \in \ww$ such that $g(m) = f(\max N_m)$.
    For every $m \in \omega$, since $\varphi(0) = 0 \leq m$, the set $N_m$ is nonempty. 
    Moreover, because $\varphi$ is increasing and unbounded, there exists $n \in \omega$ such that $\varphi(k) > m$ for all $k \geq n$.
    Therefore, $N_m$ is finite and $\max N_m$ is well defined.
    As a result, the function $g$ is also well defined.
    By assumption then, there exists an $a \in \mathcal{F}$ such that $|a \cap g(m)| \leq m$ for all $m \in \omega$.
    Let $n \in \omega$.
    By definition, $n \in N_{\varphi(n)}$ and so $n \leq \max N_{\varphi(n)}$.
    Since $f$ is strictly increasing by assumption, $f(n) \leq f(\max N_{\varphi(n)}) = g(\varphi(n))$, and, as a result, $|a \cap f(n)| \leq |a \cap g(\varphi(n))| \leq \varphi(n)$.
\end{proof}

\bigskip 
The ideal of countable sets is denoted by $\countableideal$. We assume that the reader is familiar with  \emph{Lebesgue measurability} and the \emph{property of Baire}, as well as the ideals $\nullideal$ and $\meagreideal$ of \emph{null} and \emph{meagre} subsets of $\dw$, respectively.   


\bigskip
Given $x,y \in \dw$, we write $x+y \in \dw$ for the real such that $(x+y)(n) = x(n)+y(n) \mod{2}$. 
For   $X,Y \subseteq \dw$, we write $X+Y = \set{x+y \st x \in X, y \in Y}$.


\begin{Def} A set $X \subseteq \dw$ is called 
    
    \begin{enumerate}
        \item  \emph{null-additive} if $X+Y \in \nullideal$  for every $Y \in \nullideal$, and
        \item \emph{meagre-additive} if $X+Y \in \meagreideal$   for every $Y \in \meagreideal$.
    \end{enumerate} We use $\nulladditiveideal$ and $\meagreadditiveideal$ to denote the collections of null-additive and meagre-additive sets, respectively.
\end{Def}

It is easy to see that $\nulladditiveideal$ and $\meagreadditiveideal$ form $\sigma$-ideals and that $\countableideal \subseteq \nulladditiveideal \subseteq \nullideal$ and $\countableideal \subseteq \meagreadditiveideal \subseteq \meagreideal$. A result of Shelah  \cite{Shelah1995} shows that $\nulladditiveideal \subseteq \meagreadditiveideal$. We also recall the well-known ideal $\stronglynullideal$  of \emph{strong-measure-zero} sets, and  the Galvin-Mycielski-Solovay characterization:

\begin{Thm}[Galvin, Mycielski, Solovay {\cite{GalvinMycielskiSolovay1979}}]
    \label{thm stronglynull characterization}
     $X \in \stronglynullideal$  if and only if for every meagre $Y \subseteq \dw$, $X+Y \neq \dw$.
\end{Thm}

By analogy, the dual collection $\stronglymeagreideal$ of \emph{strongly meagre} sets can be defined as follows:
\begin{Def} \label{dd}
    A set $X \subseteq \dw$ is \emph{strongly meagre} if for every null $Y \subseteq \dw$, $X+Y \neq \dw$.
\end{Def}
 

It is easy to see that $\nulladditiveideal \subseteq \stronglymeagreideal$ and $\meagreadditiveideal \subseteq \stronglynullideal$.
The statements $\stronglynullideal = \countableideal$ and $\stronglymeagreideal = \countableideal$ are known as the \emph{Borel Conjecture} ($\BC$) and \emph{Dual Borel Conjecture} ($\dualBC$), respectively.

\bigskip

Finally we introduce the main combinatorial object of our interest:   the \emph{Raisonnier filter}.  We will use the following definition given below;  a number of alternative definitions is  available, see \cite[Theorem 4.6.13]{BaJu95}.

\begin{Def}
    Let $x,y \in \dw$ be such that $x \neq y$. The \textit{splitting point} of $x,y$ is defined as  $h(x,y) = \min\set{ n \in \omega \st x(n) \neq y(n) }$. Given a set $X \subseteq \dw$, we define the set of \textit{splitting points of $X$} as $ H(X) = \set{h(x,y) \st x,y \in X, \, x \neq y}$.
\end{Def}

\begin{Def} [The Raisonnier Filter] \label{Raisonnier filter definition}
    Let $X \subseteq \dw$. Then  $\Raisonnier{X} \subseteq \dw$ is defined as the set of all $a \in \dw$, such that there exists a countable collection $\{Y_n \mid  n<\omega\}$  with  $X \subseteq \bigunion{Y_n}{n<\omega}{}$ and $H(Y_n) \subseteq a$ for all $n$. 
\end{Def}

It is easily  seen  that this defines a filter on $\omega$ extending the Fr\'echet filter.

\begin{Lem} $\F_X$ is trivial (i.e.,  contains all sets) if and only if $X$ is countable. \end{Lem}

\begin{proof} If $X$ is countable then it can be covered by singletons of the form $\{x\}$, and by definition $H(\{x\}) = \varnothing$, hence $\varnothing \in \F_X$.  Conversely,   suppose $\varnothing \in \F_X$.  Then $X$ can be covered by   $\bigcup_n Y_n$ with $H(Y_n) =\varnothing$.  But this means that each $Y_n$ is a singleton, since any set $Y_n$ that contains  two  reals must have $H(Y_n) \neq \varnothing$.   \end{proof}

The essential feature of the Raisonnier Filter in descriptive set theory comes from the observation that if  $X \subseteq \dw$ is a $\ProjectiveSigma{n}$, set, then $\Raisonnier{X}$ is $\ProjectiveSigma{n+1}$ (see \cite[p.~247]{BaJu95}). 
In Raisonnier's original proof, the role of this filter is in the form of the following theorem.
\begin{Thm}[Raisonnier, {\cite{RaisonnierFilter}}] \label{thm Raisonnier's theorem}
    If every $\ProjectiveSigma{2}$ set is Lebesgue measurable and $\aleph_1^{\L[a]} = \aleph_1$, then  $\Raisonnier{\dw \cap \L[a]}$ is a non-trivial rapid filter. In particular, $\Raisonnier{\dw \cap \L[a]}$ is a $\SIGMA^1_3$ non-trivial rapid filter, hence not all $\SIGMA^1_3$ sets are Lebesgue Measurable.
\end{Thm}

\section{The Raisonnier filter generated by simple sets} \label{sectionEASY}

The elements of the Raisonnier filter $\Raisonnier{X}$ contain the splitting points of all countable covers of a set $X \subseteq \dw$. 
As $X$ is a cover of itself, we have that $H(X) \in \Raisonnier{X}$.
However, since $\Raisonnier{X}$ extends the Fr\'echet filter, it also contains the sets $\set{k \in H(X) \st k \geq n}$ for every $n \in \omega$.
Furthermore, countable partitions of $X$ are also countable covers of it.
Because of this, if it is possible to define a partition  $\{X_n \mid n<\omega\}$ of $X$ that avoids some splitting points $K \subseteq H(X)$ in the sense that $K \cap \bigunion{H(X_n)}{n<\omega}{} = \varnothing$, then $H(X) \setminus K \in \Raisonnier{X}$ as well.

Therefore, the Raisonnier filter of a set $X \subseteq \dw$ can give us an indication of its size, specifically in regard to its splitting points. The intuition is the following:  for ``large'' sets of reals, it is not possible to avoid many of their splitting points by using countable partitions.
As a result, $\Raisonnier{X}$ will contain only ``dense'' sets of naturals, and will be closer to the Fr\'echet filter.
On the other hand, if a set $X$ is ``small'' and there are countable partitions of it avoiding many of its splitting points, then $\Raisonnier{X}$ will be large and contain many ``sparse'' sets of naturals.

\bigskip The purpose of this section is to look more closely at basic properties of $\Raisonnier{X}$ and how they depend on properties of $X$, confirming the intuition mentioned above.


 
 \begin{Lem} \label{prop Raisonnier filter subset}
    If   $X \subseteq Y$, then $\Raisonnier{Y} \subseteq \Raisonnier{X}$.
\end{Lem}
\begin{proof}
    A countable cover of a set is also a countable cover of its subsets.
\end{proof}

Next, we observe the following: suppose we are given a countable collection of sets and their countable covers. To obtain a countable cover of their union, it is enough to take the union of the covers.
The splitting points of the new cover will naturally be the union of the splitting points of the individual covers.
Therefore, the Raisonnier filter of a countable union can be generated as given below.
\begin{Prop} \label{prop Raisonnier filter union}
    Let  $X = \bigcup_{n<\omega} X_n$. Then 
    \[
        \Raisonnier{X} = \set*{ \bigunion{a_n}{n<\omega}{}  \; \st \;    a_n \in \Raisonnier{X_n} \text{ for all $n$} } \,.
    \]
\end{Prop}
\begin{proof}
     Consider an arbitrary collection $\set{a_n \in \Raisonnier{X_n} \st n < \omega}$. For every $n$ there exists a countable cover $\{X_{n,m} \mid m < \omega\}$ of $X_n$ such that $a_n \supseteq \bigunion{H(X_{n,m})}{m<\omega}{}$. Let $a = \bigunion{a_n}{n<\omega}{}$. We have that
    \[
        X = \bigunion{X_n}{n<\omega}{} \subseteq 
        \bigunion{\bigunion{X_{n,m}}{m<\omega}{}}{n<\omega}{} \quad \text{and} \quad
        a = \bigunion{a_n}{n<\omega}{} \supseteq 
        \bigunion{\bigunion{H(X_{n,m})}{m<\omega}{}}{n<\omega}{} \; \text{,}
    \]
    meaning that $a \in \Raisonnier{X}$.    
    On the other hand, let $a \in \Raisonnier{X}$. For any $n < \omega$ we have that $a \in \Raisonnier{X} \subseteq \Raisonnier{X_n}$, so letting $a_n=a$ for all $n$, we get $a_n \in \Raisonnier{X_n}$ and  $a = \bigunion{a_n}{n<\omega}{}$.
\end{proof}

\bigskip
It is easy to see that  any finite number of splitting points can be avoided  when covering a subset of $\dw$. Thus, one expects  $\Raisonnier{X}$ to be insensitive to finite modifications of the reals of $X$. This is confirmed by the following property.

\begin{Prop} \label{prop Raisonnier filter eventual}
    For every $X \subseteq \dw$, $\Raisonnier{X} = \Raisonnier{X+\rationals}$.
\end{Prop}
\begin{proof}
    Let $X \subseteq \dw$. 
    We have that $\Raisonnier{X+\rationals} \subseteq \Raisonnier{X}$ and it suffices to prove the other direction.
    Let $a \in \Raisonnier{X}$ such that $a \supseteq \bigunion{H(Y_n)}{n<\omega}{}$, where $X \subseteq \bigunion{Y_n}{n<\omega}{}$. Consider the following collection of sets: for $n < \omega$ and $s,t \in \dlw$ we define
    \[
        Z_{n,s,t} = \set{x \in \dw \st x \in [s], \, \exists y \in Y_n \cap [t] \, \forall m \geq |t| \, (x(m) = y(m))} \,.
    \]
    Let $x \in X+\rationals$, in which case there exists a $y \in X$ such that $x \eqev y$. 
    Since $y \in X$, there exists an $n < \omega$ such that $y \in Y_n$. Moreover, because $x \eqev y$, there exists an $N \in \omega$ such that $x(m) = y(m)$ for all $m \geq N$. Let $s = x\restriction N$ and $t = y\restriction N$. 
    By definition, $x \in [s]$, $y \in Y_n \cap [t]$, $|s| = |t| = N$ and $x \in Z_{n,s,t}$. 
    Therefore,
    \[
        X+\rationals \subseteq
        \bigcup \set{Z_{n,s,t} \st n < \omega, \, s,t \in \dlw, |s| = |t|} \,.
    \]

    Let now $n, s,t \in \dlw$ be such that $|s| = |t| = N$ and $k \in H(Z_{n,s,t})$. By definition, there exist $x,y \in Z_{n,s,t}$ with $x \neq y$ and $h(x,y) = k$. Because $x,y \in [s]$, it must be that $k \geq |s|$. By definition, there exist $x',y' \in Y_n\cap[t]$ such that $x(m) = x'(m)$ and $y(m) = y'(m)$ for all $m \geq N$. Since $x',y'$ share the same initial segment $t$ up to $N$ and are equal to $x$ and $y$ respectively from $N$ onward, whose splitting point comes after $N$, we have that $k = h(x,y) = h(x',y') \in H(Y_n)$. Therefore, $H(Z_{n,s,t}) \subseteq H(Y_n)$. As a result,
    \[
        a \supseteq \bigunion{H(Y_n)}{n<\omega}{} \supseteq
        \bigcup\set{H(Z_{n,s,t}) \st n<\omega, \, s,t \in \dlw, \, |s| = |t|}
    \]
    and thus $a \in \Raisonnier{X+\rationals}$. 
\end{proof}
\bigskip

Another intuitive property of the Raisonnier filter concerns shifting reals to the right in the sense of subsets of $\omega$. It is clear that appending the same finite sequence in front of all reals of a set $X \subseteq \dw$ would shift their splitting points by the length of that sequence.
As a result, the Raisonnier filter of this new set would simply consist of the elements in the original filter shifted accordingly.
\begin{Def}
    If $a \subseteq \omega$ and $n \in \omega$ we will write $a+n$ for the set $\set{k+n \st k \in a}$.
    If $X \subseteq \dw$ and $s \in \dlw$, we will write $s \concat X$ for the set $\set{s\concat x \st x \in X}$. 
\end{Def}
\begin{Prop} \label{prop Raisonnier concat}
    If $X \subseteq \dw$ and $s \in \dlw$, then $\Raisonnier{s \concat X}$ is generated by the set $\set{a+|s| \st a \in \Raisonnier{X}}$.
\end{Prop}
\begin{proof}
    If $a \in \Raisonnier{X}$ is witnessed by the countable cover $\bigunion{Y_n}{n< \omega}{} \supseteq X$ with $a \supseteq \bigunion{H(Y_n)}{n<\omega}{}$, then it is easy to see that $\bigunion{(s \concat Y_n)}{n<\omega}{} \supseteq s \concat X$ and $a + |s| \supseteq \bigunion{H(s \concat Y_n)}{n<\omega}{}$.

    On the other hand, let $b \in \Raisonnier{s\concat X}$ such that $b \supseteq \bigunion{H(Y_n)}{n<\omega}{}$, where $s\concat X \subseteq \bigunion{Y_n}{n<\omega}{}$. 
    For every $n < \omega$, let $Z_n = \set{x \in X \st s \concat x \in Y_n}$. If $x \in X$, then $s \concat x \in Y_n$ for some $n<\omega$ and $x \in Z_n$. 
    Therefore, $X \subseteq \bigunion{Z_n}{n<\omega}{}$ and $\bigunion{H(Z_n)}{n<\omega}{} \in \Raisonnier{X}$. 
    If now $k \in H(Z_n)$, then there exist $x,y \in X$ such that $s \concat x, s \concat y \in Y_n$, $x \neq y$ and $h(x,y) = k$. We then have that $k+|s| = h(s\concat x, s \concat y) \in H(Y_n)$. As a result, $\left( \bigunion{H(Z_n)}{n<\omega}{} \right)+|s| \subseteq b$. 
\end{proof}
\smallskip

To find Raisonnier filters with more concrete characterizations, we can turn to closed sets.
The splitting points of a closed set correspond exactly to the splitting points of the tree defining it.
Moreover, (uncountable) closed sets are so large that any countable cover must eventually contain all the splitting points of the corresponding tree.
\smallskip
\begin{Def}
    If $T \subseteq \dlw$ is a tree and $s \in T$, then let $T_s = \set{t \in T \st s \subseteq t \text{ or } t \subseteq s}$ and $H(T) = \set{h(s,t) \st s,t \in T, \; s \neq t}$.
\end{Def}
\begin{Lem} \label{lem Raisonnier filter of closed sets}
    Let $T \subseteq \dlw$ be a tree and let $a \in \Raisonnier{[T]}$. Then, for every $s \in T$ there exists a $t \in T$ with $s \subseteq t$ such that $a \supseteq H(T_t)$.
\end{Lem}
\begin{proof}
    Let $a \in \Raisonnier{[T]}$ and suppose towards contradiction that there exists an $s \in T$ such that for every $t \in T$ with $s \subseteq t$ we have $H(T_t) \not\subseteq a$. By assumption, there exists a countable cover $\set{X_n \st n < \omega}$ of $[T]$ such that $a \supseteq \bigunion{H(X_n)}{n<\omega}{}$.
    We can then inductively define a sequence $s_0 \subseteq s_1 \subseteq 
    \dots$ in $T$  such that  $X_n \cap [s_{n+1}] = \emptyset$ for every $n$, as follows: 
    
    First, let $s_0 = s$. Then, assume that we have defined $s_n \in T$. Since  $s_n \supseteq  s$, there exists $k \in H(T_{s_n}) \setminus a$ by assumption. 
    As $k \in H(T_{s_n})$ is a splitting point, we have that $k \geq |s_n|$. Moreover, there exist $x,y \in [T_{s_n}]$ such that $h(x,y) = k$. Assume  without loss of generality that $x(k) = 0$ and $y(k) = 1$, and let $t = x\restriction k = y \restriction k$.
    Because $k \notin a \supseteq H(X_n)$, we have two cases:
    \begin{enumerate}
        \item For every $z \in X_n$ with $t \subseteq z$, $z(k) = 0$. In that case,  we can define $s_{n+1} = y \restriction (k+1)$, giving us that $s_{n+1} \in T$ and $s_{n} \subset s_{n+1}$. Moreover, for any $w \in [s_{n+1}]$ we have $w(k) = y(k) = 1$ and so $w \notin X_n$.
        \item For every $z \in X_n$ with $t \subseteq z$, $z(k) = 1$. In that case we can similarly define $s_{n+1} = x \restriction (k+1)$.
    \end{enumerate}

    Let $w = \bigunion{s_n}{n<\omega}{}$. Then by construction  $w \in [T_{s}]$ but $w \notin X_n$ for any $n$, contradicting our initial assumption. 
\end{proof}

When the tree in question is uniform in the sense that whether a node is splitting depends only on its length,  we can obtain a precise description of its Raisonnier filter.
\medskip
\begin{Def}
    A tree $T \subseteq \dlw$ is called \emph{uniform} if for all $s,t \in T$ with $|s|=|t|$, we have that $s$ is splitting if and only if $t$ is splitting. Note that this is equivalent to stating that the splitting nodes $t$ of $T$ are precisely those with $|t| \in H(T)$.\footnote{These are not necessarily ``uniform perfect trees'' in the sense of Silver forcing.}
\end{Def}
\begin{Prop} \label{prop Raisonnier filter uniform closed}
    For any uniform tree $T \subseteq \dlw$, $\Raisonnier{[T]} = \set{a \subseteq \omega \st H(T) \almostsubseteq a}$.
\end{Prop}
\begin{proof}
    First, observe that $H(T) \in \Raisonnier{[T]}$ and that, because $\Raisonnier{[T]}$ extends the Fr\'echet filter, $\set{a \subseteq \omega \st H(T) \almostsubseteq a} \subseteq \Raisonnier{[T]}$. For the other direction, if $a \in \Raisonnier{[T]}$, by Lemma \ref{lem Raisonnier filter of closed sets} there exists an $s \in T$ such that $H(T_s) \subseteq a$.
    Since the tree is uniform, $H(T) \almostsubseteq H(T_s) \subseteq a$.
\end{proof}


\begin{Cor}\label{cor Raisonnier filter of powerset}
    For every $a \subseteq \omega$, $\Raisonnier{\powerset(a)} = \set{b \subseteq \omega \st a \almostsubseteq b}$. 
\end{Cor}
\begin{Cor} \label{cor Cantor Raisonnier filter}
    $\Raisonnier{\dw}$ is the Fr\'echet filter.
\end{Cor}



Proposition \ref{prop Raisonnier filter uniform closed} also has immediate consequences concerning open subsets of the Cantor space. 
 
\begin{Cor} \label{Raisonnier open corollary}
    The following hold:
    \begin{enumerate}
        \item For every $s \in \dlw$, $\Raisonnier{[s]}$ is the Fr\'echet filter.
        \item For every nonempty open set $X \subseteq \dw$, $\Raisonnier{X}$ is the Fr\'echet filter. 
        \item For every $X \subseteq \dw$ with non-empty interior,  $\Raisonnier{X}$ is the Fr\'echet filter. \label{Raisonnier interior}
    \end{enumerate}
\end{Cor}
\begin{proof}
    The first statement follows from the fact that $[s] = s\concat \dw$, using Corollary \ref{cor Cantor Raisonnier filter} and Proposition \ref{prop Raisonnier concat}. 
    The other two statements follow directly from Proposition \ref{prop Raisonnier filter subset}.
\end{proof}

We note that the converse of the above Corollary does not hold. For example, if $a = \set{2n \st n \in \omega}$ and $b = \set{2n+1 \st n \in \omega}$,  the set $X = \powerset(a) \cup \powerset(b)$ has empty interior.
By Propositions \ref{prop Raisonnier filter union}, \ref{prop Raisonnier filter subset} and Corollary \ref{cor Raisonnier filter of powerset}, $\Raisonnier{X}$ is a subset of both
$\set{c \subseteq \dw \st a \almostsubseteq c}$ and $\set{c \subseteq \dw \st b \almostsubseteq c}$.
As a result, it can only contain cofinite sets and is thus equal to the Fr\'echet filter.

\bigskip A more interesting application of Lemma \ref{lem Raisonnier filter of closed sets} leads to the following conclusion about the notion of an \emph{unbounded filter} from Definition \ref{def:unbounded}. 



\begin{Thm} \label{thm Raisonnier filter of perfect sets}
    If $T$ is a perfect tree, then $\Raisonnier{[T]}$ is a bounded  filter.
\end{Thm}
\begin{proof} We will use condition (4) of  Theorem \ref{thm unbounded filter characterization} and  define a partitioning real $d \in \partitioningreals$ bounding the elements of $\Raisonnier{[T]}$. 

For $s \in T$, let $\eta(s)$ denote the first splitting node of $T$ extending $s$. Let $S_0 = \set{\emptyset}$ and $d(0) = 0$. Having  defined $S_n$ and $d(n)$, inductively define
    \begin{itemize}
        \item $S_{n+1} = \set{\eta(s) \st s \in T, \; |s| = d(n)}$, and
        \item $d(n+1) = \max\set{|s| \st s \in S_{n+1}} + 1$.
    \end{itemize}
    Clearly each $S_n$ is a finite set, so this construction is well-defined.

    Now let $a \in \Raisonnier{[T]}$ be given. We want to show that  $a \cap I^d_n \neq \varnothing$ for all except finitely many $n$. By Lemma \ref{lem Raisonnier filter of closed sets}, there exists $s \in T$ such that $a \supseteq H(T_s)$.
    Let $k \in \omega$ be smallest such that $d(k) \leq |s| < d(k+1)$.
    For any $n > k$, there exists an $s' \in T$ with $s' \supseteq s$ and $|s'| = d(n)$.
    By definition, $\eta(s') \in S_{n+1}$ and therefore $d(n) \leq |s'| \leq \eta(s') < d(n+1)$.
    Furthermore, for $t_0 = \eta(s')\concat 0$ and $t_1 = \eta(s') \concat 1$ we have $t_0,t_1 \in T_s$.
    As a result, $h(t_0,t_1) = |\eta(s')| \in H(T_s) \subseteq a$ and $a \cap I^d_n \neq \emptyset$.
\end{proof}

This theorem has  interesting descriptive-theoretic consequences.
\begin{Cor} \label{cor Raisonnier filter of analytic sets}
The following hold:
\begin{enumerate}
    \item If $X \subseteq \dw$ has the perfect set property, then $\Raisonnier{X}$ is either trivial or bounded.
    \item For every analytic $X \subseteq \dw$, $\Raisonnier{X}$ is either trivial or bounded.
    \item $(\ZF)$ If all sets of reals have the perfect set property then all Raisonnier filters are trivial or bounded.
\end{enumerate}
\end{Cor}

\bigskip In the next section we will investigate unbounded Raisonnier filters more closely. This section is concluded with a curious observation regarding the Raisonnier filter of the constructible reals.
\begin{Prop} \label{prop Raisonnier filter constructibles}
    If $b \in \Raisonnier{\dw \cap L}$ and $n \in \omega$, then $b+n \in \Raisonnier{\dw \cap L}$.
\end{Prop}
\begin{proof}
    Let $X = \dw \cap L$. 
    It is easy to see that for any $n \in \omega$ we can write $X = \bigcup\{s \concat X \mid s \in 2^n\}$.
    By Proposition \ref{prop Raisonnier filter union}, we then have that
    \[
        \Raisonnier{X} = \set*{\bigunion{b_s}{s \in 2^n}{} \st \forall s \in 2^n \, (b_s \in \Raisonnier{s \concat X})} \,.
    \]
    By Proposition \ref{prop Raisonnier concat}, for every $s \in 2^n$, $\Raisonnier{s \concat X}$ is the filter $\mathcal{G}$ generated by $\set{b+n \st b \in \Raisonnier{X}}$. Therefore, we can write $\Raisonnier{X}$ as
    \[
        \Raisonnier{X} = \set*{\bigunion{b_s}{s \in 2^n}{} \st \forall s \in 2^n \, (b_s \in \mathcal{G})} \,.
    \]
    As a result, if $b \in \Raisonnier{X}$, then $b + n \in \mathcal{G}$ and thus $b+n \in \Raisonnier{X}$.
\end{proof}
In the case that $|\dw \cap L| = \aleph_1$ and thus $\Raisonnier{\dw \cap L}$ is not the Fr\'echet filter, it follows that $\Raisonnier{\dw \cap L}$  will contain some increasingly ``sparse'' elements.
This is because, for every $a \in \Raisonnier{\dw \cap L}$, $a + n \in \Raisonnier{\dw \cap L}$ as well and so their intersection $a \cap (a+n) = \set{k \in a \st k+n \in a}$ must also be in $\Raisonnier{\dw \cap L}$.

\section{Unbounded Raisonnier Filters} \label{sectionUNBOUNDED}

In the previous section, it was shown that sets containing perfect sets are large enough to have bounded Raisonnier filters.
On the other hand, we can expect that sets whose Raisonnier filter is unbounded are small. 
In fact, they form a $\sigma$-ideal.

\begin{Def} \label{def unbounded Raisonnier ideal}
    Let $\unboundedideal = \set{X \subseteq \dw \st \text{$\Raisonnier{X}$ is unbounded}}$.
\end{Def}
\begin{Thm}
    The collection $\unboundedideal$ forms a $\sigma$-ideal.
\end{Thm}
\begin{proof}
    First, $\emptyset \in \unboundedideal$ holds trivially, and closure under subsets follows directly from Proposition \ref{prop Raisonnier filter subset}.
    In order to prove closure under countable unions, let $\set{X_i \st i<\omega} \subseteq \unboundedideal$ and $X = \bigunion{X_i}{i<\omega}{}$. 
    We will show that $\Raisonnier{X}$ is unbounded using condition (3) of Theorem \ref{thm unbounded filter characterization}.
    For this, we will need the following lemma.

    \begin{Lem*}
        For every increasing $f \in \ww$ there exists a sequence $\set{a_i \in \Raisonnier{X_i} \st i<\omega}$ such that for every $n \in \omega$
        \begin{equation*}
            \label{lemma property}
            \tag{\textasteriskcentered}
            \existinf k \in \omega:\, \left( \bigunion{a_i}{i \leq n}{} \right) \cap [k,f(k)] = \emptyset \ . 
        \end{equation*}
    \end{Lem*}
    \begin{proof}[Proof of lemma]
        Let $f \in \ww$ be increasing. We will define the sequence inductively. 
        First, because $\Raisonnier{X_0}$ is unbounded, there exists $a_0 \in \Raisonnier{X_0}$ satisfying \eqref{lemma property} for $n = 0$.

        Assume now that we have defined $a_0,\dots,a_n$ so that \eqref{lemma property} is satisfied for some $n \in \omega$.
        Let
        \[
            b = \set*{k \in \omega \st \left( \bigunion{a_i}{i \leq n}{} \right) \cap [k,f(k)] = \emptyset} \;,
        \]
        which is infinite by the inductive hypothesis, and let $\hat{b}$ be its increasing enumeration.
        Let $g = f \circ \hat{b}$, which is an increasing function, since both $f$ and $\hat{b}$ are increasing.
        By assumption, $\Raisonnier{X_{n+1}}$ is unbounded and there exists an $a_{n+1} \in \Raisonnier{X_{n+1}}$ such that
        \[
            c = \set*{k \in \omega \st a_{n+1} \cap [k,g(k)] = \emptyset}
        \]
        is infinite.
        
        Let now $k \in \hat{b}[c]$.
        Because $k \in b$, we have that $(\bigunion{a_i}{i\leq n}{}) \cap [k,f(k)] = \emptyset$.
        Moreover, $k = \hat{b}(l)$ for some $l \in c$, for which we know that $a_{n+1} \cap [l,g(l)] = \emptyset$.
        As $\hat{b}$ is strictly increasing however, $l \leq \hat{b}(l) = k$, and therefore
        \[
            [k,f(k)] \subseteq [l,f(\hat{b}(l))] = [l,g(l)] \;,
        \]
        which gives us $a_{n+1} \cap [k,f(k)] = \emptyset$.
        As a result, $(\bigunion{a_i}{i\leq n+1}{}) \cap [k,f(k)] = \emptyset$ as well.
        Since $|c| = \aleph_0$ and $\hat{b}$ is injective, $\hat{b}[c]$ is infinite and $a_0,\dots,a_{n+1}$ satisfy \eqref{lemma property}, which completes the inductive step.
    \end{proof}

    Returning to the main proof, let $f \in \ww$ be an increasing function and let $\set{a_i \in \Raisonnier{X_i} \st i< \omega}$ be the sequence given by the Lemma.
    We inductively define the following strictly increasing sequence $N = \set{n_m \in \omega \st m < \omega}$:
    \begin{itemize}
        \item First, let $n_0$ be the minimum $k \in \omega$ such that $a_0 \cap [k, f(k)] = \emptyset$, and then,
        \item having defined $n_m$, let $n_{m+1}$ be the minimum $k \in \omega$ such that $k > n_m$, $k > f(n_m)$ and $(\bigunion{a_i}{i\leq m+1}{}) \cap [k,f(k)] = \emptyset$, which exists by the Lemma.

    \end{itemize}
    Finally, define $c_0 = a_0$ and for every $i<\omega$ define $c_{i+1} = \set{k \in a_{i+1} \st k > f(n_m)}$.
    Because each $\Raisonnier{X_i}$ extends the Fr\'echet filter, $c_i \in \Raisonnier{X_i}$ for every $i<\omega$ and so $c = \bigunion{c_i}{i<\omega}{} \in \Raisonnier{X}$ by Proposition \ref{prop Raisonnier filter union}.
    Let $m<\omega$ and let
    \[
        K = \left( \bigunion{c_i}{i\leq m}{} \right) \cap [n_m, f(n_m)] \quad \text{and} \quad 
        K' = \left( \bigunion{c_i}{i > m}{} \right) \cap [n_m, f(n_m)] \;.
    \]
    By construction, we have that $K \subseteq (\bigunion{a_i}{i\leq m}{}) \cap [n_m, f(n_m)] = \emptyset$ and, for every $i > n$, by the fact that $N$ is an increasing sequence, 
    \[
        c_i \cap [n_m, f(n_m)] \subseteq 
        \set{k \in \omega \st k > f(n_m)} \cap [0,f(n_m)] = \emptyset \,
    \]
    meaning that $K' = \emptyset$ as well.
    Therefore, $c \cap [n_m, f(n_m)] = K \cup K' = \emptyset$.
    As a result, we get $c \cap [n,f(n)] = \emptyset$ for every $n \in N$ and the proof is complete.
\end{proof}

When a set $X \subseteq \dw$ has an unbounded Raisonnier filter, given some partition of $\omega$ into intervals, we can find a countable cover of $X$ without splitting points in infinitely many intervals of that partition.
Because of this, we can eventually predict the values of any $x \in X$ on those intervals, given the initial segment of $x$. We recall that there is a very similar characterization of the meagre-additive ideal $\meagreadditiveideal$ due to Shelah \cite{Shelah1995}, see also {\cite[Theorem 2.7.17]{BaJu95}}.

\begin{Thm}[Shelah, 1995] \label{thm characterization of meager-additive}
    The following are equivalent for every $X \subseteq \dw$:
    \begin{enumerate}
        \item $X \in \meagreadditiveideal$, and
        \item for every $f \in \partitioningreals$ there exists a $g \in \partitioningreals$ and a real $y \in \dw$ such that
        \[
            \forall x \in X\:\forallbutfin n \in \omega \; \exists k \in \omega \; (I^f_k \subseteq I^g_n \ \  \text{and} \ \ x \restriction I^f_k = y \restriction I^f_k) \;.
        \]
    \end{enumerate}
\end{Thm}
In this statement, condition (2) can be understood as stating   that, given any $f \in \partitioningreals$ we can find a real $y$ that ``predicts''   the values of any $x \in X$ in infinitely many $f$-intervals. 
Moreover, we have a bound, in the form of $g \in \partitioningreals$, as to how far apart those intervals can at most be located (except for finitely many initial places). This intuition leads to the following result. 
\medskip

\begin{Thm}
    $\unboundedideal \subseteq \meagreadditiveideal$
\end{Thm}

\begin{proof}  Let $X \in \unboundedideal$. Let $f \in \partitioningreals$ be given. We will find $g$ and $y$ satisfying condition (2) in Theorem \ref{thm characterization of meager-additive}.

\p 
First, let $d$ be an increasing sequence defined inductively by

 \vbox{\begin{itemize}
\item $d(0) =0,$
\item if  $d(n)$ has been defined,  let $d(n+1)$ be the least such that 
$$d(n) < f(k) < f(k + n \cdot 2^{d(n)})< d(n+1)$$
where $k$ is  least such that $d(n)<f(k)$.  \end{itemize}}

\p The purpose of this definition is to ensure the following:

\begin{center}
$(*) \;\;\; $ there are $n \cdot 2^{d(n)}$-many intervals of the type $I^f_k$ contained  in each  $I^d_n$.
\end{center}

\p Using one of the  properties of an unbounded filter,  let $a \in \Raisonnier{X}$ be such that $\exists^\infty n  \: (a \cap I^d_n = \varnothing)$.  This means we can find  $n_0 < n_1 < n_2< \dots$ and a countable cover $\{X_j \mid j<\omega\}$, such that $X \subseteq \bigcup_{j<\omega}X_j$, and 
\[
\forall j \forall i \: (H(X_j) \cap I^d_{n_i} \; = \; \varnothing)
\]  
We also define $g(0) = 0$ and $g(i) = d(n_i)$ for every $i > 0$, that is, the beginning of the interval $I^d_{n_i}$. 

\p To construct the predicting real $y \in \dw$, proceed as follows: on intervals $I^d_{n}$ for $n \neq n_i$,  the value is irrelevant.  Within an interval $I^d_{n_i}$, note that since $H(X_j) \cap I^d_{n_i} = \varnothing$,  for each $\sigma \in 2^{d(n_i)}$  there can be at most one $\tau \in 2^{d(n_i+1)}$ extending $\sigma$  such that $\tau \in X_j \till 2^{d(n_i+1)}$.  Call it $\tau_{(\sigma,n_i)}$ if it exists.

\p By condition $(*)$ there are $n_i \cdot 2^{d(n_i)}$-many intervals of the form $I^f_k$ inside $I^d_{n_i}$ so,  by going over all the relevant $\tau_{(\sigma,\ell)}$, we can define a predicting real  $y$ so that for every $\sigma \in 2^{d(n_i)}$ and every $\ell < n_i$ we have $$y \till I^f_k \; = \; \tau_{(\sigma,\ell)} \till I^f_k$$ for some $k$. This completes the construction.

\p Now we check that this works.  Let $x \in X$, let $j$ be such that $x \in X_j$, and let $K$ be large enough so that $j \leq n_K$.  We claim that for  $i > K$ there is $k$ such that $I^f_k \subseteq I^g_i$ and $x \till I^f_k = y \till I^f_k$,  which will complete the proof.  So for $i>K$,  look at the interval $I^d_{n_i}$, and let $\sigma = x \till n_i$.  Since $x \in X_j$,  it follows that $x \till d(n+1)$ is precisely $\tau_{(\sigma,j)}$ as defined in the above construction.  Since also $j < n_i$,  our construction guarantees that for some $k$ we have  $$x \till I^f_k = \tau_{(\sigma,j)} \till I^f_k = y \till I^f_k,$$ and also $I^f_k \subseteq I^d_{n_i} \subseteq I^g_i$. This completes the proof.
\end{proof}

\begin{Cor} If  the \emph{Borel Conjecture}   holds, then $\unboundedideal = \countableideal$.  
\end{Cor}

\begin{proof} This follows from $\unboundedideal \subseteq \meagreadditiveideal \subseteq \stronglynullideal$.  \end{proof}

Since Raisonnier filters generated by countable sets are trivial, it follows that under $\BC$ all Raisonnier filters are either trivial or bounded.

\section{Rapid Raisonnier Filters}  \label{sectionMAIN}

In this section, we return to Raisonnier's original proof and provide a natural generalization of it in terms of \emph{binary slaloms}.
As it turns out, the crucial step in the proof can be reversed, providing a characterization of the rapidity of the filter $\Raisonnier{X}$ in terms of $X$ being captured by such binary slaloms.
This in turn tells us that sets that generate a rapid Raisonnier filters are precisely the null-additive sets. We will introduce the following definition in analogy to the previous section:

\begin{Def}
    Let $\rapidideal = \set{X \subseteq \dw \st \text{$\Raisonnier{X}$ is rapid}}$.
\end{Def}

We first recall some notion related to Bartoszy\'{n}ski's characterization of additivity and cofinality of the null ideal in terms of  \emph{slaloms} (see \cite{BartAdd} or  \cite[pp.~50-53]{BaJu95}).

\begin{Def}[Bartoszy\'nski] \label{def slalom} Let $\varphi: \omega \to \omega$ be any unbounded function.  A function $S: \omega \to [\omega]^{<\omega}$ is called a \emph{$\varphi$-slalom} if $|S(n)| \leq  \varphi(n)$ for all $n$. When $\varphi(n)=n$,  we  refer to  $\varphi$-slaloms simply as \emph{slaloms}. If  $x \in \ww$ is a real, then we say that \emph{$x$ goes through $S$},  denoted by $x \in^*S$, if $f(n) \in S(n)$ for all but finitely many $n$.  If $X \subseteq \ww$ is a set of reals, then we say that \emph{$X$ goes through $S$}, or that \emph{$S$ captures $X$}, if  $x \in^*S$ for every $x \in X$.  
\end{Def}

In \cite{BartAdd} Bartoszy\'{n}ski proved that  $\add(\N)$  is the least cardinality of a set $X \subseteq \ww$ that does not go through any slalom,  and $\cof(\N)$  is the least number of slaloms necessary for every real in $\ww$ to go through at least one of them. This  combinatorial argument also has a descriptive set theoretic component.

\begin{Thm}[Solovay/Bartoszy\'{n}ski] \label{thm slalom equivalence}
The following are equivalent:
\begin{enumerate}
\item All $\SIGMA^1_2$ sets of reals are Lebesgue measurable.
\item For all $a \in \ww$ there are measure-one-many random reals over $\L[a]$.
\item For all $a \in \ww$ there exists a slalom $S$ such that $\ww \cap \L[a] $ goes through $ S$.
\end{enumerate} \end{Thm}

\begin{proof} The equivalence between (a) and (b) is Solovay's well-known characterization theorem implicit in \cite{SolovayAModel}; see \cite[Theorem 9.3.1]{BaJu95} for a clear proof. The equivalence between (b) and (c) is a direct adaptation of the  combinatorial principles used to prove the characterization of $\add(\N)$ in terms of slaloms. See also \cite[Theorem 2.3.11(1)]{BaJu95}.
\end{proof}



We can define an analogous notion of slaloms for the Cantor space.



\begin{Def}[Binary Slaloms]
    Let $d \in \partitioningreals$ and $\varphi \in \ww$. A function $B: \omega \to \powersetfin{\dlw}$ is called a \textit{$d$-binary $\varphi$-slalom} if for every $n$:
    \begin{enumerate}
        \item $|B(n)| \leq \varphi(n)$, and
        \item $B(n) \subseteq \{0,1\}^{{|I^d_n|}}$.
    \end{enumerate}
    When  $\varphi(n) = n$ we will simply call them \emph{$d$-binary slaloms}.
    
   \p  If $B$ is a $d$-binary ($\varphi$-)slalom and $x \in \dw$ then we say that  \textit{$x$ goes through $B$},   and write $x \goesthrough B$, if   $x \restriction{I^d_n} \in B(n)$ for all but finitely many $n \in \omega$. Likewise we say that a set  $X \subseteq \dw$ \textit{goes through} $B$, or that  $B$ \textit{captures $X$},   if $x \in^* B$ for every $x \in X$.

    \p We let $\binaryslaloms{d}{\varphi}$ denote the set of all $d$-binary $\varphi$-slaloms. When $\varphi(n) = n$, we will simply write $\binaryslaloms{d}{}$.
\end{Def}
Binary slaloms place restrictions on the possible segments of reals in the Cantor space, in the same way that slaloms place restrictions on the digits of reals in the Baire space.
With this in mind, we can consider  a natural translation mapping from $\dw$ to $\ww$ which allows us to translate between the two types of slaloms.

\begin{Def}$\;$ 
 \begin{enumerate}
\item
    Let $\nat: \dlw \to \omega$ be the function mapping a finite sequence in $\dlw$ to the natural number represented by the sequence in the binary system. 
   \item Given a partitioning real $d \in \partitioningreals$, let $\nat_d: \dw \to \ww$ be the function defined as $\nat_d(x)(n) = \nat(x\restriction{I^d_n})$ for all $x \in \dw$ and $n \in \omega$.
   \item Let $\prod_d$ abbreviate  $\prod_{n\in \omega}2^{|I^d_n|} \subseteq \ww$, i.e., the set of all $x \in \ww$ such that $ x(n) < 2^{d(n+1)-d(n)}$ for all $n$. \end{enumerate}
\end{Def}

It is clear that for any partitioning real $d$, $\nat_d$ is a bijection between $\dw$ and $\prod_d$ (a compact subset of $\ww$). The reverse function $\nat_d^{-1}$ is well-defined on $\prod_d$ and can be naturally interpreted as the map that sends  every $x \in \prod_d$ to the concatenation of binary representations of the values $x(n)$ (with additional $0$'s if necessary), written in sequence.  See Figure \ref{fig:nat example} for an example.


\begin{figure}[h!]
    \centering
    \begin{tikzpicture}[scale=0.95, transform shape]
        \node (S0) {$x\restriction I^d_0$};
        \node (S1) [right =24.05pt of S0] {$x \restriction I^d_1$};
        \node (S2) [right =15.7pt of S1] {$x \restriction I^d_2$};
        \node (S3) [right =41pt of S2] {$x \restriction I^d_3$};
        \node (S4) [right =24.2pt of S3] {$x \restriction I^d_4$};
        \node (S5) [right =7.1pt of S4] {$x \restriction I^d_5$};
        \node (more) [right =21.3pt of S5] {};
        \node (S0set) [below =0pt of S0] {
            \begin{tabular}{|c c c c|}
            \hline
            $1$ & $1$ & $0$ &$0$ \\
            \hline
            \end{tabular}
        };
        \node (S1set) [below =0pt of S1] {
            \begin{tabular}{|c c|}
            \hline
            $0$ &$1$\\ 
            \hline
            \end{tabular}
        };
        \node (S2set) [below =0pt of S2] {
            \begin{tabular}{|c c c|}
            \hline
            $0$&$0$&$1$\\
            \hline
            \end{tabular}
        };
        \node (S3set) [below =0pt of S3] {
            \begin{tabular}{|c c c c c|}
            \hline
            $1$&$0$&$1$&$1$&$1$\\
            \hline
            \end{tabular}
        };
        \node (S4set) [below =0pt of S4] {
            \begin{tabular}{|c|}
            \hline
            $0$\\
            \hline
            \end{tabular}
        };
        \node (S5set) [below =0pt of S5] {
            \begin{tabular}{|c c c|}
            \hline
            $0$&$1$&$0$\\
            \hline
            \end{tabular}
        };
        \node (moreset) [below =5.5pt of more] {
            \begin{tabular}{| c}
                 \hline
                 $\dots$\\ 
                 \hline
            \end{tabular}
        };
        \node (natx contents) at (5, -3) {
            \begin{tabular}{|c|c|c|c|c|c|c}
                \hline
                 $12$ & $1$ & $1$ & $23$ & $0$ & $2$ & $\dots$ \\ 
                \hline
            \end{tabular}
        };
        \node (natx) [left =0pt of natx contents] {$\nat_d(x)= $};
        \node (x) [left =-3pt of S0set] {$x=$};
        \draw[->,thick] (0.25,-0.9) -- (2.75,-2.65);
        \draw[->,thick] (1.75,-0.9) -- (3.6,-2.65);
        \draw[->,thick] (3.3,-0.9) -- (4.2,-2.65);
        \draw[->,thick] (5.65,-0.9) -- (4.95,-2.65);
        \draw[->,thick] (7.45,-0.9) -- (5.7,-2.65);
        \draw[->,thick] (8.65,-0.9) -- (6.25,-2.65);
    \end{tikzpicture}
    \caption{A real $x \in \dw$ translated into $\nat_d(x) \in \ww$, using the partitioning real $d = \left<0,4,6,9,14,15,18,\dots\right>$.}
    \label{fig:nat example}
\end{figure}


The mapping $\nat_d$  translates $d$-binary slaloms to regular slaloms. With a slight abuse of notation, we will write the following:

\begin{Def} If $B$ is a $d$-binary slalom then  $\nat_d[B]$ denotes the slalom $S$ defined by $S(n) = \{\nat(\sigma) \mid \sigma \in B(n)\}$. Conversely, if $S$ is any  slalom,
 then  $\nat_{d}^{-1}[S]$ denotes the $d$-binary  slalom $B$ defined by $B(n) = \nat^{-1}[S(n)] = \{\sigma \in \{0,1\}^{|I^d_n|} \mid \nat(\sigma) \in S(n)\}$.
 
\end{Def}

 The following lemma then follows immediately:

 \vbox{
 \begin{Lem} \label{easy-peasy} $\:$ \begin{enumerate}
     \item If $d \in \partitioningreals$ and $B$ is a $d$-binary slalom, then for all $x \in \dw$:
$$x \in^* B \;\; \Leftrightarrow \;\; \nat_d(x) \in^* \nat_d[B].$$ 



\item If $S$ is a slalom, then for all $d \in \partitioningreals$ and all $x \in \dw$:

$$x \in^* \nat_d^{-1}[S] \;\; \Leftrightarrow \;\; \nat_d(x) \in^* S$$

 \end{enumerate}
 
\end{Lem}}

A less familiar theorem of Shelah connects the null-additive ideal $\nulladditiveideal$ to binary slaloms in the following way:
\begin{Thm}[Shelah, 1995 \cite{Shelah1995}, see also {\cite[Th.~2.7.18]{BaJu95}}]
    The following are equivalent:
    \begin{enumerate}
        \item $X \in \nulladditiveideal$, and
        \item for every $d \in \partitioningreals$ there exists a $d$-binary slalom $B$ such that $X$ goes through $B$.
    \end{enumerate}
\end{Thm} 


We now return to Raisonnier's original proof  showing that if all $\SIGMA^1_2$ sets are measurable then $\Raisonnier{\dw \cap L[a]}$ is rapid, and note that a simplified version of this proof, presented, e.g., in \cite[pp.474--475]{BaJu95} proceeds directly via Theorem \ref{thm slalom equivalence} and the existence of slaloms capturing the reals of $L[a]$. So the following theorem is implicitly proven.

\begin{Thm}[Raisonnier/Bartoszy\'{n}ski/Shelah] \label{lem binary slaloms to rapid Raisonnier filter}
    Let $X \subseteq \dw$. If for every partitioning real $d \in \partitioningreals$, $X$ goes through a $d$-binary slalom, then $\Raisonnier{X}$ is rapid. In particular $\nulladditiveideal \subseteq \rapidideal$.
\end{Thm}
\begin{proof}
    Let $d \in \partitioningreals$. By our assumption, there exists a $B \in \binaryslaloms{d}{}$ such that
    \[
        \forall x \in X \; \forallbutfin n \in \omega \; (x \restriction I^d_n \in B(n)) \,.
    \]
    From this, we obtain an $a \in \Raisonnier{X}$ exactly as in the proof presented in \cite[pp. 474--475]{BaJu95}, so that $|a \cap d(n)| \leq n^5$ for all $n \in \omega$, proving that $\Raisonnier{X}$ is rapid by Lemma \ref{lem rapid filter alternative definition}.
\end{proof}

The above proposition essentially works because a restriction on the possible segments of reals $x \in X$ directly leads to a restriction on the eventual splitting points necessary in covers of $X$. 
This implication can be reversed: restricting the possible splitting points of a set also leads to a restriction of the possible segments of its elements. This is the main idea in our next result.

\begin{Thm} \label{lem rapid Raisonnier filter to binary slalom}
    Let $X \subseteq \dw$. If $\Raisonnier{X}$ is rapid and $d \in \partitioningreals$, then $X$ goes through a $d$-binary slalom. In particular $\rapidideal \subseteq \nulladditiveideal$.
\end{Thm}
\begin{proof}
    Let $X \subseteq \dw$ be such that $\Raisonnier{X}$ is rapid and let $d \in \partitioningreals$.
    By the rapidity of $\Raisonnier{X}$, there must exist an $a \in \Raisonnier{X}$ such that $|a \cap d(n)| \leq \chi(n)$ for all $n \in \omega$, where
    \[
        \chi(n) =
        \begin{cases}
            0, \, \text{if $n \leq 1$}\\
            \floor{\log_2(\floor{\sqrt{n-1}})}, \, \text{otherwise} \;.
        \end{cases}
    \] 
    By definition then, there exists a countable cover $\{X_n \mid n<\omega\}$ of $X$ such that $a \supseteq \bigunion{H(X_n)}{n<\omega}{}$.
   We  define a collection $\{B_n \mid n<\omega\}$ of $d$-binary $\varphi$-slaloms, where $\varphi(m) = \floor{\sqrt{m}}$ for all $m \in \omega$.
    For any $n < \omega$ let $B_n(0) = \varnothing$ and, for $m \in \omega$ with $m > 0$, let
    \[
        B_n(m) = X_n \restriction I^d_m \,.
    \]
    This is the set of all possible finite sequences of length $|I^d_m|$ that can occur as the $I^d_m$-sections of reals in $X_n$.
    Any two $x,y \in X_n$ can only diverge at a splitting point $k \in H(X_n)$.
    This means that the set $\set{x \restriction [0,d(m+1)) \st x \in X_n}$, and so $B_n(m)$, can have a cardinality of at most $2^k$, where $k$ is the number of splitting points of $X_n$ below $d(m+1)$.
    As a result, we have that
    \[
        |B_n(m)| \leq 2^{|H(X_n) \cap d(m+1)|}
        \leq 2^{|a \cap d(m+1)|} \leq 2^{\floor{\log_2(\floor{\sqrt{m}})}{}} \leq \sqrt{m}
    \]
    for all $n,m \in \omega$ with $m > 0$.
    
    Let now $B: \omega \to \powersetfin{\dlw}$ such that for every $m \in \omega$, $B(m) = \bigunion{B_n(m)}{n<\sqrt{m}}{}$.
    We have that
    \[
        |B(m)| \leq \sum_{n<\sqrt{m}} |B_n(m)| \leq
        \sum_{n<\sqrt{m}} \sqrt{m} \leq \sqrt{m}^2 = m
    \]
    Moreover, if $s \in B(m)$, then there exists some $n < \sqrt{m}$ such that $s \in B_n(m)$, in which case, by construction, $s \in \set{0,1}^{|I^d_m|}$.
    Therefore, $B$ is a $d$-binary slalom. Finally, if $x \in X$, then there exists an $n < \omega$ such that $x \in X_n$. 
    Then, for every $m \in \omega$ with $m > 0$, $x\restriction I^d_m \in B_n(m)$ by definition and so, for every $m > n^2$, $x\restriction I^d_m \in B_n(m) \subseteq B(m)$.
    This means that $x \goesthrough B$ and we have thus shown that $X$ goes through the $d$-binary slalom $B$.
\end{proof}


\begin{Thm} \label{thm Raisonnier rapid filter characterization}
   The following are equivalent for every $X \subseteq \dw$:
    \begin{enumerate}
        \item $\Raisonnier{X}$ is rapid,
        \item for every $d \in \partitioningreals$, $X$ goes through a $d$-binary slalom,
        \item for every $d \in \partitioningreals$, $\nat_d[X]$ goes through a slalom, and
        \item $X \in \nulladditiveideal$.
    \end{enumerate}
\end{Thm}
\begin{proof}
    For the equivalence between (2) and (3), note that if $B$ is a $d$-binary slalom so that $X$ goes through $B$, then by Lemma \ref{easy-peasy} (1), $\nat_d[X]$ goes through $\nat_d[B]$. Conversely, if $\nat_d[X]$ goes through a slalom $S$, then by Lemma \ref{easy-peasy} (2), $X$ goes through $\nat_d^{-1}[S]$.
\end{proof}

\begin{Cor}
    $\rapidideal = \nulladditiveideal$.
\end{Cor}


\begin{Cor} If either the \emph{Borel Conjecture} or the \emph{Dual Borel Conjecture} hold then $\rapidideal=\countableideal$. In particular, the only rapid Raisonnier filters are the trivial ones.
\end{Cor}

\begin{proof} This follows because $\nulladditiveideal \subseteq \meagreadditiveideal \subseteq \stronglynullideal$ and $\nulladditiveideal \subseteq \stronglymeagreideal$.\end{proof}

\begin{Cor}\label{cor NA set implies rapid filter}
    If there exists an uncountable $X \in \nulladditiveideal$, then there exists a non-trivial rapid filter.
    Moreover, if there is a $\ProjectiveSigma{n}$ uncountable $X \in \nulladditiveideal$, then there exists a $\ProjectiveSigma{n+1}$ non-trivial rapid filter.
\end{Cor}

\begin{Cor} Suppose $\dw\cap \L[a] \in \nulladditiveideal$ and $\dw \cap \L[a]$ is uncountable. Then there is a $\SIGMA^1_3$ rapid filter.
\end{Cor}

We also obtain an alternative proof of the following result from  \cite[Theorem 4.1]{JSSouslin}, which was shown by  a different argument.

\begin{Cor} \label{516} $\BC \; +\; $``every $\SIGMA^1_2$ set is Lebesgue measurable'' $\; \to \; \forall a (\omega_1^{L[a]}<\omega_1)$.
\end{Cor}

\begin{proof} If  every $\SIGMA^1_2$ set is Lebesgue measurable, then using Theorem \ref{thm slalom equivalence}, for all $a \in \ww$ and $d \in \partitioningreals$ there is a slalom capturing  $\ww \cap L[a,d]$, which in particular implies point (3) of Theorem \ref{thm Raisonnier rapid filter characterization}, and therefore $\dw \cap L[a] \in \nulladditiveideal \subseteq \stronglynullideal=\countableideal$. 
\end{proof}

Finally,  we note that Shelah's result \cite{Shelah1995} about the inclusion $\nulladditiveideal \subseteq \meagreadditiveideal$ now also has an alternative proof, namely via the following chain of inclusions, 

$$\nulladditiveideal=\rapidideal\subseteq \unboundedideal\subseteq \meagreadditiveideal$$
where the middle inclusion is  because rapid filters are unbounded.


\section{A   converse to Raisonnier's theorem} \label{sectionCONVERSE}

Since Theorem \ref{thm Raisonnier rapid filter characterization} equates a subset of $\dw$ having a rapid Raisonnier filter to its going through a slalom for every $d \in \partitioningreals$, it is natural to ask whether a converse to Raisonnier's Theorem \ref{thm Raisonnier's theorem} can be obtained.
We can indeed establish a partial converse with an additional assumption.

In the proof of Raisonnier's theorem, reals in $\dw \cap \L[a]$ are essentially encoded by elements of $\ww \cap \L[a,d]$ for some $d \in \partitioningreals$. 
Here, we would conversely like to encode the reals in $\ww \cap \L[a]$ by  reals in $\dw \cap \L[a,d]$.
However, for this to work, we need a preliminary upper bound on the values $f(x)$ for an arbitrary $f \in \ww \cap \L[a]$,  so that a suitable partitioning real $d  \in \partitioningreals$ may be chosen so that $f \in \prod_d$ and so that the inverse value $\nat_d^{-1}(f)$ is well-defined. This is possible if there exists a  \emph{dominating real} over  $\ww \cap \L[a]$ for all $a$.



\begin{Lem} \label{lem rapidideal and domination to slaloms}
    Assume that
    \begin{enumerate}
        \item for every $a \in \ww$, $\Raisonnier{\dw \cap \L[a]}$ is rapid, and 
        \item for every $a \in \ww$ there exists a dominating real over  $\ww \cap \L[a]$.
    \end{enumerate}
    Then, for every $a \in \ww$ there exists a slalom $S$ such that $\ww \cap \L[a]$ goes through $S$.
\end{Lem}
\begin{proof}



    Let $a \in \ww$.  Let $d$ be a real dominating $\ww \cap \L[a]$.  Without loss of generality we may assume that $d(n) > 0$ for all $n \in \omega$.   Let $\partreal{d} \in \partitioningreals$ be defined by  $$\partreal{d}(n) := \sum_{m < n}d(m).$$ Then $\partreal{d}$ is a partitioning real,  and by assumption $\Raisonnier{\dw \cap \L[a,d]}$ is rapid, so by Theorem \ref{thm Raisonnier rapid filter characterization}  there exists a slalom $S$ such that $\nat_{\partreal{d}}[\dw \cap \L[a,d]]$ goes through $S$.  We claim that also $\ww \cap \L[a]$ goes through $S$.
    
    To see this, let  $f \in \ww \cap \L[a]$. There exists $g \in \ww$ such that $f \eqev g$ and $g(n) \leq d(n)$ for all $n \in \omega$.   Since $f$ and $g$ differ in only finitely many values, $g \in \L[a] \subseteq \L[a,d]$ as well.     Because $g(n) \leq d(n)$,  we can write $g(n)$ in the binary system using at most $d(n)$ digits.    In other words, we have  $$g(n) \leq d(n) < 2^{d(n)} = 2^{|{I^{\partreal{d}}_n}|}$$ for all $n \in \omega$. Therefore $g$ is contained in $ \Pi_{\tilde{d}}$ and so $\nat^{-1}_{\partreal{d}}(g)$ is well-defined.  Moreover,  since $\partreal{d} \in \L[a,d]$ we also have $\nat^{-1}_{\partreal{d}}(g) \in \L[a,d]$.  In particular, $g  =\nat_{\partreal{d}}(\nat_{\partreal{d}}^{-1}(g)) \in^* S$.  Because $f \eqev g$, it must be that $f \goesthrough S$ as well.
\end{proof}

Combining the above with the Solovay/Bartoszy\'{n}ski characterization  from Theorem \ref{thm slalom equivalence} gives rise to the following characterization:
\begin{Thm} \label{thm rapid ideal and domination characterization of Sigma12 measurability}
    The following are equivalent.
    \begin{enumerate}
        \item Every $\ProjectiveSigma{2}$ set is Lebesgue measurable.
        \item For every $a \in \ww$ there exists a dominating real over $\L[a]$ and  $\Raisonnier{\dw \cap \L[a]}$ is rapid.
    \end{enumerate}
\end{Thm}
\begin{proof}
    Assuming (1) immediately leads to $\ww \cap \L[a]$ being dominated by some real for every parameter $a \in \ww$. It also leads to $\Raisonnier{\dw \cap \L[a]}$ being rapid for every $a \in \ww$ by Raisonnier's Theorem \ref{thm Raisonnier's theorem}.
    The other direction is Lemma \ref{lem rapidideal and domination to slaloms}.
\end{proof}

Because of the equivalence between $X \subseteq \dw$ having a rapid Raisonnier filter and it being null-additive, this theorem can be seen as a descriptive set-theoretic version of Pawlikowski's result that $\add(\nullideal) = \min\set{\boundingnum, \non(\nulladditiveideal)}$ {\cite[Lemma 2.3]{Pawlikowski1985}}, which is proven in a similar fashion.
\medskip

Theorem \ref{thm rapid ideal and domination characterization of Sigma12 measurability} is particularly interesting in view of the general study of \textit{transcendence properties}  and \textit{regularity properties}.  Recall that a \emph{Laver tree} is a tree in $\wlw$ where every node extending a stem is $\omega$-splitting. A set $A \subseteq \ww$ is said to be \emph{Laver measurable} if for every Laver tree $T$ there is a Laver subtree $S$ such that $[S] \subseteq A$ or $[S] \cap A = \varnothing$.    
In \cite{BrLo99} Brendle and L\"owe prove that the transcendence property  ``$\forall  a \in \ww$ there is a dominating real over $\ww \cap \L[a]$'' is equivalent to the statement ``all $\SIGMA^1_2$ sets are Laver measurable'', so we obtain an immediate corollary.


\begin{Cor} \label{CorDescriptive}
    The following are equivalent:
    \begin{enumerate}
        \item Every $\ProjectiveSigma{2}$ set is Lebesgue measurable.
        \item Every $\ProjectiveSigma{2}$ set is Laver measurable and $\forall a \in \ww (\Raisonnier{\dw \cap \L[a]}$ is rapid$)$.
    \end{enumerate}
\end{Cor} 

Alternatively, point 2 above can also be phrased as follows  \emph{``Every $\Sigma^1_2$ set is Laver measurable and $\forall a   (\dw \cap \L[a] \in \nulladditiveideal)$''} or as follows  $$\forall a   (\dw \cap \L[a] \in \mathcal{K}_\sigma \cap \nulladditiveideal)$$ where $\mathcal{K}_\sigma$ refers to the ideal of $\sigma$-compact subsets of $\ww$.

\section{Some consequences related to the Borel Conjecture} \label{sectionBOREL}

In this paper we have come across several ideals with the following inclusion relations:

$$\countableideal \; \subseteq \; \rapidideal = \nulladditiveideal \; \subseteq \; \unboundedideal  \; \subseteq \;  \meagreadditiveideal  \; \subseteq \;  \stronglynullideal$$
\vspace{-0.7cm}$$\rotatebox{-90}{$\subseteq$} \hspace{1.3cm}$$
$$\stronglymeagreideal\hspace{1.3cm}$$
These ideals are related to the \emph{Borel} and \emph{dual Borel Conjectures} and are consistently equal to the countable ideal. We can introduce the following notation:  

\begin{Def} Let $\BC(\mathcal{I})$ denote the statement $\I = \countableideal$.\footnote{Although this is not a ``conjecture'' but a statement known to be independent of $\ZFC$, we still choose to use the abbreviation ``$\BC$'' to emphasize the connection with the Borel Conjecture. Also, note that $\I$ does not need to be an ideal, e.g., $\stronglymeagreideal$ is consistently not an ideal by \cite{BartoszynskiShelah2001}.}\end{Def}

Several mathematicians have looked at the relationship between statements $\BC(\I)$ and whether they can be separated. It is known that $\BC(\meagreadditiveideal)$ is equivalent to $\BC = \BC(\stronglynullideal)$, see \cite[Corollary 4.3]{Calderon}. In \cite{JSW}, a model is produced for  $\BC + \lnot \dualBC$, while the Cohen model satisfies \mbox{$\dualBC + \lnot \BC(\meagreadditiveideal)$} (see \cite{Carlson}). In a recent preprint, Cie\'{s}lak \cite{Cieslak2026} shows that $\BC(\nulladditiveideal) + \lnot \dualBC + \lnot \BC(\meagreadditiveideal)$ holds in the model obtained by iterating Miller and Random forcing, and also that $\BC(\nulladditiveideal) + \lnot \BC(\meagreadditiveideal)$ holds in the iterated Miller model. Thus, all statements in the diagram below, except $\BC(\unboundedideal)$, can consistently be separated.

$$\BC = \BC(\stronglynullideal)= \BC(\meagreadditiveideal)  \;\; \Rightarrow \;\;  \BC(\unboundedideal)  \;\; \Rightarrow \;\; \BC(\nulladditiveideal)= \BC(\rapidideal) $$
$$\hspace{6.7cm}\rotatebox{90}{$\Rightarrow$}$$
$$\hspace{6.7cm}\dualBC = \BC(\stronglymeagreideal)$$

\bigskip
We do not know how to separate $\BC(\unboundedideal)$ from the other $\BC$-statements, nor whether it holds or fails, e.g., in the Cohen or Miller models.

 However, we have several new observations that follow from $\rapidideal=\nulladditiveideal$.  

\begin{Lem} \label{lem no rapid} Suppose there are no non-trivial rapid filters. Then $ \BC(\nulladditiveideal)$ holds.
\end{Lem} 

\begin{proof}  If $X \in \nulladditiveideal = \rapidideal$ then $\Raisonnier{X}$ is a rapid filter. This must be the trivial filter, so $X$ is countable.
\end{proof}

For example,  the result from \cite{Cieslak2026} that $\BC(\nulladditiveideal)$ holds in the Miller model can alternatively be proved using Lemma \ref{lem no rapid} and the fact from \cite{MillerRational}   that there are no (non-trivial) rapid filters in the Miller model.  

The converse of Lemma \ref{lem no rapid}  is not true. For example, in the Cohen model $\dualBC$ and thus $\BC(\nulladditiveideal)$ are true by \cite{Carlson}, but  also rapid filters exist by \cite[Lemma 4.6.3.(2)]{BaJu95}.

\bigskip Recall that in Corollary \ref{516} we mentioned that the statement ``all $\SIGMA^1_2$ sets are Lebesgue measurable'' together with $\BC$ leads to an inaccessible. The same holds for any of the $\BC$-statements mentioned above, specifically:

\begin{Cor} $\BC(\nulladditiveideal)\; + \;$``all $\SIGMA^1_2$ sets are Lebesgue measurable'' $\; \to \; \forall a (\omega_1^{L[a]}<\omega_1)$.
\end{Cor}

\begin{proof} If all $\SIGMA^1_2$ sets are measurable and $\exists a (\omega_1^{L[a]}=\omega_1)$, then $\dw \cap L[a]$ is an uncountable set in $\rapidideal$ by Raisonnier's proof, hence $\BC(\nulladditiveideal)$ fails.\end{proof} 

\bigskip Another question indirectly answered by our results relates to Martin's Axiom: in \cite[Theorem 4.6.16]{BaJu95} the authors prove that  $\mathsf{MA}(\text{$\sigma$-centered})$ implies that $\Raisonnier{X}$ is unbounded for $|X|<\cont$ and $\mathsf{MA}(\text{$\sigma$-linked})$ implies that $\Raisonnier{X}$ is rapid for $|X|<\cont$. It is asked as an open question whether $\mathsf{MA}(\text{$\sigma$-centered})$ implies that $\Raisonnier{X}$ is rapid. However,  
$\mathsf{MA}(\text{$\sigma$-centered})$ is known to be consistent with $\dualBC$ by \cite{JudahShelahDBC},  which we now know to imply $\BC(\rapidideal)$ and therefore that the only rapid Raisonnier filters are the trivial ones, making the answer to the question negative.

\bigskip Finally, in view of results like Corollary \ref{cor Raisonnier filter of analytic sets}, it  seems interesting to explore definable versions of Borel Conjecture-type statements. We are not aware of previous such studies, but we can consider the following definition:

\begin{Def} For a projective pointclass $\boldsymbol{\Gamma}$, let $\GAMMA(\BC(\I))$ denote the following statement: if $X$ is a $\GAMMA$-set in $\I$, then it is countable.
\end{Def}

 The statement $\GAMMA(\BC(\nulladditiveideal))$ is equivalent to saying that for all $X$ in $\GAMMA$, the Raisonnier filter $\Raisonnier{X}$ is either trivial or not rapid;  the same with $\GAMMA(\BC(\unboundedideal))$ and ``unbounded''.

\begin{Lem} We have the following: $\;$

\begin{enumerate}
    \item $\GAMMA(\textup{Perfect Set Property}) \; \Rightarrow \; \GAMMA(\BC(\unboundedideal))  \; \Rightarrow \; \GAMMA(\BC(\nulladditiveideal))$.

    \item If there is no $\SIGMA^1_{n+1}$ rapid filter then $\SIGMA^1_n(\BC(\nulladditiveideal))$

    \item If there is no $\SIGMA^1_{n+1}$ unbounded filter then $\SIGMA^1_n(\BC(\unboundedideal))$ and hence $ \SIGMA^1_n(\BC(\nulladditiveideal))$.
    
\end{enumerate}
 
\end{Lem}

Thus in particular,   $\SIGMA^1_1(\BC(\unboundedideal))$ and $\SIGMA^1_1(\BC(\nulladditiveideal))$ are true in $\ZFC$. If $\forall a (\omega_1^{L[a]}<\omega_1)$, then $\SIGMA^1_2(\BC(\unboundedideal))$ and $\SIGMA^1_2(\BC(\nulladditiveideal))$ are true. In the Solovay model and Shelah's model for projective Baire property without inaccessibles, $\SIGMA^1_n(\BC(\unboundedideal))$ and $\SIGMA^1_n(\BC(\nulladditiveideal))$ hold for all $n$.

\section{Open Questions} \label{sectionQUEST}

 The investigations in this paper have shed light on  the combinatorics of the Raisonnier filter, its relation to binary slaloms, Lebesgue measurability, and the the ideals $\nulladditiveideal$ and $\meagreadditiveideal$.

\bigskip While most ideals encountered here have been investigated before, we do not know much about $\unboundedideal$, so some natural questions are the following:

\begin{Question}  $\;$ \begin{enumerate}[(a)]

    \item Is there a combinatorial characterization of $\unboundedideal$ that does not involve the Raisonnier filter? For example, is it equivalent to a property using Slaloms, or to an ideal similar to $\nulladditiveideal$ and $\meagreadditiveideal$?
    \item What can be said about its cardinal invariants such as $\add(\unboundedideal)$, $\cov(\unboundedideal)$, etc.?
    \end{enumerate}
\end{Question}

This question is closely related to the Borel Conjecture with respect to the ideal $\unboundedideal$, so we can ask the following questions:

\begin{Question}  $\;$ \begin{enumerate}[(a)]
    \item Can $\BC(\unboundedideal)$ be separated from $\BC(\meagreadditiveideal)$ and $\BC(\nulladditiveideal)$?

    \item Does $\BC(\unboundedideal)$ hold in the Cohen, Hechler and Miller models?
\end{enumerate}
\end{Question}

Recall that Corollary  \ref{CorDescriptive} provides an interesting equivalence between two projective regularity properties, and a third  ``transcendence'' property ``$\forall a \left(\dw \cap \L[a]\right) \in \nulladditiveideal$''. It would be interesting to find a regularity property that could replace the above statement and give rise to a theorem exclusively about regularity properties on the second level.

\begin{Question} \label{question is there regularity} Does there exist a regularity property $\mathsf{P}$ such that the statement ``all $\SIGMA^1_2$ sets have $\mathsf{P}$'', or ``all $\DELTA^1_2$ sets have $\mathsf{P}$'', is    equivalent to the statement ``$\forall a \left(\dw \cap \L[a]\right) \in \nulladditiveideal$''? \end{Question}

\bibliographystyle{plain}
\bibliography{Khomskii_Master_Bibliography}{}

\end{document}

%% file: newcommands.tex
\newcommand{\p}{\medskip \noindent}

\newcommand{\PP}{\mathscr{P}}
\newcommand{\cont}{{2^{\aleph_0}}}

\newcommand{\till}{{\upharpoonright}}

\newcommand{\ZF}{{\rm \sf ZF}}
\newcommand{\ZFC}{{\rm \sf ZFC}}

\newcommand{\SIGMA}{\boldsymbol{\Sigma}}
\newcommand{\DELTA}{\boldsymbol{\Delta}}

\newcommand{\ww}{\omega^\omega}
\newcommand{\wlw}{\omega^{<\omega}}
\newcommand{\dw}{2^\omega}
\newcommand{\dlw}{2^{<\omega}}

\newcommand{\F}{\mathcal{F}}

\newcommand{\s}[1]{\medskip \noindent \textbf{#1}}

\newcommand{\N}{\mathcal{N}}

\newcommand{\cov}{{\rm cov}}
\newcommand{\add}{{\rm add}}
\newcommand{\non}{{\rm non}}
\newcommand{\cof}{{\rm cof}}

\newtheorem{Thm}{Theorem}[section]

\newtheorem{Lem}[Thm]{Lemma}
\newtheorem*{Lem*}{Lemma}
\newtheorem{Cor}[Thm]{Corollary}
\newtheorem{Prop}[Thm]{Proposition}

\newtheorem*{Claim*}{Claim}

\theoremstyle{definition}
\newtheorem{Def}[Thm]{Definition}

\newtheorem{Remark}[Thm]{Remark}
\newtheorem{Note}[Thm]{Note}

\newtheorem{Question}[Thm]{Question}

\newcommand{\I}{\mathcal{I}}